\documentclass[10pt]{amsart}
\usepackage{mathrsfs,amssymb,amsmath,url}

\theoremstyle{plain}
    \newtheorem{thm}{Theorem}
    \newtheorem{lem}[thm]{Lemma}
    \newtheorem{prop}[thm]{Proposition}
    \newtheorem{cor}[thm]{Corollary}

    \newtheorem{prob}[thm]{Problem}

\theoremstyle{definition}

\theoremstyle{remark}

\newcommand{\nin}{\notin}
\DeclareMathOperator{\Inv}{Inv}

\DeclareMathOperator{\Cl}{Cl}
\DeclareMathOperator{\Cll}{Cl_{loc}}
\DeclareMathOperator{\Monl}{Mon_{loc}}
\DeclareMathOperator{\Grl}{Gr_{loc}}
\DeclareMathOperator{\Clp}{Cl_{fpart}}
\DeclareMathOperator{\Monp}{Mon_{fpart}}

\newcommand{\inv}{\Inv}

\newcommand{\um}{^{(m)}}
\newcommand{\uo}{^{(1)}}

\newcommand{\cll}[1]{\langle #1 \rangle_{loc}}
\newcommand{\cl}[1]{\langle #1 \rangle}

\renewcommand{\O}{{\mathscr O}}
\newcommand{\On}{{\mathscr O}^{(n)}}

\newcommand{\Oo}{{\mathscr O}^{(1)}}

\DeclareMathOperator{\ppol}{Pol}
\DeclareMathOperator{\id}{id}
\DeclareMathOperator{\ran}{ran} \DeclareMathOperator{\dom}{dom}
\newcommand{\rest}{\upharpoonright}

\newcommand{\C}{{\mathscr C}}
\newcommand{\D}{{\mathscr D}}
\newcommand{\F}{{\mathscr F}}

\newcommand{\To}{\rightarrow}
\newcommand{\mult}{\times}
\newcommand{\GG}{{\mathscr G}}
\newcommand{\R}{{\mathscr R}}
\newcommand{\X}{{\mathfrak X}}

\newcommand{\G}{{\mathfrak G}}
\renewcommand{\H}{{\mathfrak H}}
\newcommand{\M}{{\mathscr M}}

\renewcommand{\S}{{\mathscr S}}

\renewcommand{\L}{{\mathfrak L}}

    \title{More sublattices of the lattice of local clones}

    \author{Michael Pinsker}
    \email{marula@gmx.at}
    \urladdr{http://dmg.tuwien.ac.at/pinsker/}
    \address{Laboratoire de Math\'{e}matiques Nicolas Oresme\\
CNRS UMR 6139\\ Universit\'{e} de Caen\\14032 Caen Cedex\\
France}
    \thanks{The author is grateful for support through Erwin Schr\"{o}dinger Fellowship J2742-N18 of the Austrian Science Fund (FWF)}


\begin{document}
    \maketitle
    \begin{abstract}
        We investigate the complexity of the lattice of local clones over
        a countably infinite base set. In particular, we prove
        that this lattice contains all algebraic lattices with at
        most countably many compact elements as complete
        sublattices, but that the class of lattices embeddable into
        the local clone lattice is strictly larger than that: For example, the lattice $M_{2^\omega}$ is a sublattice of the local clone lattice.
    \end{abstract}

%--------------------------------------------------------------------The problem--------------------------

    \section{Local clones}

\subsection{Defining local clones}
        Fix a countably infinite base set $X$, and denote for all $n\geq 1$ the set $X^{X^n}=\{f: X^n\rightarrow X\}$
        of $n$-ary operations on $X$ by $\On$. Then the union
        $\O:=\bigcup_{n\geq 1}\On$ is the set of all finitary
        operations on $X$. A \emph{clone} $\C$ is a subset of $\O$
        satisfying the following two properties:
        \begin{itemize}
            \item $\C$ contains all projections, i.e., for all $1\leq
            k\leq n$ the operation $\pi^n_k\in\On$ defined by
            $\pi^n_k(x_1,\ldots,x_n)=x_k$, and

            \item $\C$ is closed under composition, i.e., whenever
            $f\in\C$ is $n$-ary and $g_1,\ldots,g_n\in\C$ are
            $m$-ary, then the operation $f(g_1,\ldots,g_n)\in\O\um$ defined by
            $$
                (x_1,\ldots,x_m)\mapsto f(g_1(x_1,\ldots,x_m),\ldots,g_n(x_1,\ldots,x_m))
            $$
            also is an element of $\C$.
        \end{itemize}

        Since arbitrary intersections of clones are again clones, the set of all clones on $X$, equipped with the order of
        inclusion, forms a complete lattice $\Cl(X)$. In this paper, we are not interested
        in all clones of $\Cl(X)$, but only in
        clones which satisfy an additional topological closure property: Equip $X$ with the discrete topology, and
        $\On=X^{X^n}$ with the corresponding product topology (Tychonoff topology), for every
        $n\geq 1$. A clone $\C$ is called \emph{locally closed} or just
        \emph{local} iff each of its $n$-ary fragments $\C\cap\On$ is a closed
        subset of $\On$. Equivalently, a clone $\C$ is local iff it satisfies the
            following interpolation property:\begin{quote}
                For all $n\geq 1$ and all $g\in\On$, if for all finite $A\subseteq X^n$
                there exists an $n$-ary $f\in\C$ which agrees with
                $g$ on $A$, then $g\in\C$.\end{quote}

        Again, taking the set of all local clones on $X$, and
        ordering them according to set-theoretical inclusion, one
        obtains a complete lattice, which we denote by $\Cll(X)$: This is because intersections of
        clones are clones, and because arbitrary intersections of
        closed sets are closed. We are interested in the structure
        of $\Cll(X)$, in particular in how complicated it is as a
        lattice.

        Before we start our investigations, we give an alternative
        description of local clones which will be useful.
        Let $f\in\On$ and let $\rho\subseteq X^m$ be a relation. We
        say that $f$ \emph{preserves} $\rho$ iff
        $f(r_1,\ldots,r_n)\in\rho$ whenever $r_1,\ldots ,r_n\in\rho$; here, the $m$-tuple
        $f(r_1,\ldots,r_n)$ is calculated \emph{componentwise}, i.e., it is the $m$-tuple whose $i$-th component is obtained by applying $f$ to the $n$-tuple consisting of the $i$-th components of the tuples $r_1,\ldots,r_n$.
        For a set of relations $\R$, we write $\ppol(\R)$ for the set
        of those operations in $\O$ which preserve all $\rho\in\R$.
        The operations in $\ppol(\R)$ are called \emph{polymorphisms}
        of $\R$. The following
        is due to \cite{Rom77}, see also the textbook\ \cite{Sze86}.

        \begin{thm}
            $\ppol(\R)$ is a local clone for all sets of relations
            $\R$. Moreover, every local clone is of this form.
        \end{thm}

        Similarly, for an operation $f\in\On$ and a relation $\rho\subseteq X^m$, we say
        that $\rho$ is \emph{invariant} under $f$ iff $f$ preserves $\rho$.
        Given a set of operations $\F\subseteq\O$, we write
        $\inv(\F)$ for the set of all relations which are invariant
        under all $f\in\F$. Since arbitrary intersections of local clones are local clones again,
        the mapping on the power set of $\O$ which assigns
        to every set of operations $\F\subseteq\O$ the smallest
        local clone $\cll{\F}$ containing $\F$ is a closure operator,
        the closed elements of which are exactly the local clones.
        Using the operators $\ppol$ and $\inv$ which connect
        operations and relations, one obtains the following well-known alternative
        for
        describing this operator (confer \cite{Rom77} or \cite{Sze86}).

        \begin{thm}
            Let $\F\subseteq\O$. Then $\cll{\F}=\ppol\Inv(\F)$.
        \end{thm}

        As already mentioned, the aim of this paper is to investigate the structure of the local clone lattice.
        So far, this lattice has been studied only sporadically, e.g.\ in \cite{RS82-LocallyMaximal}, \cite{RS84-LocalCompletenessI}. There, the
         emphasis was put on finding local completeness criteria for sets of operations $\F\subseteq \O$, i.e., on how to decide whether or
        not $\cll{\F}=\O$. Only very recently has the importance of
        the local clone lattice to problems from model theory and
        theoretical computer science been revealed:

\subsection{The use of local clones}
        Let $\Gamma=(X,\R)$ be a countably infinite structure; that is, $X$ is a
        countably infinite base set and $\R$ is a set of finitary
        relations on $X$. Consider the expansion $\Gamma'$ of $\Gamma$ by all
        relations which are first-order definable from $\Gamma$. More
        precisely, $\Gamma'$ has $X$ as its base set and its relations
        $\R'$ consist of all finitary relations which can be defined
        from relations in $\R$ using first-order formulas. A
        \emph{reduct} of $\Gamma'$ is a structure $\Delta=(X,\D)$, where $\D\subseteq
        \R'$. We also call $\Delta$ a reduct of $\Gamma$, which essentially
        amounts to saying that we expect our structure $\Gamma$ to be
        closed under first-order definitions. Clearly, the set of
        reducts of $\Gamma$ is in one-to-one correspondence with the
        power set of $\R'$, and therefore not of much interest as a
        partial order. However, it might be more reasonable to consider such
        reducts up to, say, \emph{first-order interdefinability}. That is,
        we may consider two reducts $\Delta_1=(X,\D_1)$ and $\Delta_2=(X,\D_2)$ the same iff their
        first-order expansions coincide, or equivalently iff all
        relations of $\Delta_1$ are first-order definable in $\Delta_2$ and
        vice-versa.

        In 1976, P.\ J.\ Cameron~\cite{Cameron5} showed that there are exactly five reducts of $(\mathbb Q,<)$ up to
        first-order interdefinability. Recently, M.\ Junker and M.\ Ziegler gave a
        new proof of this fact, and established that $(\mathbb Q,<,a)$, the expansion
        of $(\mathbb Q,<)$ by a constant $a$, has 116 reducts~\cite{JunkerZiegler}.
        S.\ Thomas proved that the first-order theory of the random graph also has exactly five reducts, up to
        first-order interdefinability ~\cite{RandomReducts}.

        These examples have in common that the structures under
        consideration are \emph{$\omega$-categorical}, i.e., their first-order
        theories determine
        their countable models up to isomorphism. This is no
        coincidence: For, given an $\omega$-categorical structure $\Gamma$,
        its reducts up to first-order interdefinability are in
        one-to-one correspondence with the locally closed permutation groups which contain the automorphism
        group of $\Gamma$, providing a tool for describing such reducts (confer \cite{Cam90-Oligo}).

        A natural variant of these concepts
        is to consider reducts up to \emph{primitive positive
        interdefinability}. That is, we consider two reducts
        $\Delta_1, \Delta_2$ of $\Gamma$ the same iff their
        expansions by all relations which are definable from
        each of the structures by primitive positive
        formulas coincide. (A first-order formula is called
        \emph{primitive positive} iff it is of the form $\exists
        \overline x (\phi_1 \wedge \dots \wedge \phi_l)$ for atomic formulas
        $\phi_1,\dots,\phi_l$.)  It turns out that for $\omega$-categorical structures $\Gamma$, the \emph{local
        clones} containing all automorphisms of $\Gamma$ are in one-to-one correspondence with those reducts
        of the first-order expansion of
        $\Gamma$ which are closed under primitive positive
        definitions. This recent connection, which relies on a theorem from \cite{BN06}, has already been exploited in \cite{BCP}, where the reducts of $(X,=)$, the structure whose only relation is the equality, have been
        classified using this method. It turns out that despite the simplicity of this structure, the lattice of its reducts is quite complex, and in particular has the size of the continuum.

        We mention in passing that distinguishing relational
        structures up to primitive positive interdefinability, and therefore understanding the structure of $\Cll(X)$, 			has
        recently
        gained significant importance in theoretical computer
        science, more precisely for what is known as the Constraint
        Satisfaction Problem; see \cite{JBK}, \cite{Bodirsky}, or also the introduction in \cite{BCP}.

\subsection{Main results of this paper}
		In this paper, which is the journal version of a shorter article which appeared in the conference proceedings of the ROGICS'08 conference \cite{Pin08sublatticesOfTheLocal}, we are concerned with Problem~V of the survey paper~\cite{GP07}, which asks which sublattices $\Cll(X)$ contains. We
prove that every algebraic lattice with countably many compact elements is a complete sublattice of $\Cll(X)$  (Theorem~\ref{thm:algebraicIntervals}). We also show that $\Cll(X)$ is, with respect to size, not too far from such lattices as it as a join-preserving embedding into an algebraic lattice with countably many compacts (Theorem~\ref{thm:partialClones}). All this is done in Section~2. In Section~3, we prove that the lattice $M_{2^\omega}$ embeds into $\Cll(X)$ (Theorem~\ref{thm:mcont}), thereby providing the first example of a sublattice of $\Cll(X)$ which does not embed into any algebraic lattice with countably many compacts.

We also pose a series of open problems, one in Section~2 (Problem~\ref{prob:powerOmega}), and three more in an own open problems section, Section~4 (Problems~\ref{prob:finiteHeight}, \ref{prob:Mcont}, \ref{prob:monoid}).

\subsection{Acknowledgement} I am grateful to Marina Semenova for her critical remarks which forced me to find correct proofs of the theorems.

    \section{Algebraic lattices and the local clone lattice}
	We start our investigations by observing that whereas the lattice $\Cl(X)$ of all (not necessarily local)
        clones over $X$ is algebraic, it has been discovered
        recently in~\cite{GP07} that the local clone lattice $\Cll(X)$ is far from being
        so; in fact, the following has been shown.

        \begin{prop}
            The only compact element in the lattice $\Cll(X)$ is
            the clone of projections.
        \end{prop}

	We remark that it follows from the proof of the preceding proposition given in~\cite{GP07} that $\Cll(X)$ is not even upper continuous.

        We now turn to sublattices of $\Cll(X)$. The following is a first easy
        observation which tells us that there is practically no hope that
        $\Cll(X)$ can ever be fully described, since it is believed
        that already the clone lattice over a three-element set is
        too complex to be completely understood.

        \begin{prop}\label{prop:finiteintolocal}
            Let $\Cl(A)$ be the lattice of all clones over a
            finite set $A$. Then $\Cl(A)$ is an
            isomorphic copy of an interval of $\Cll(X)$.
        \end{prop}

        \begin{proof}
            Assume without loss of generality that $A\subseteq X$.
            Assign to every operation $f(x_1,\ldots,x_n)$ on $A$ a
            set of $n$-ary operations $\S_f\subseteq\On$ on $X$ as
            follows: An operation $g\in \On$ is an element of $\S_f$
            iff $g$ agrees with $f$ on $A^n$. Let $\sigma$ map every
            clone $\C$ on the base set $A$ to the set
            $\bigcup \{\S_f: f\in \C\}$. Then the following hold:
            \begin{enumerate}
                \item
                For every clone $\C$ on $A$, $\sigma(\C)$ is a local
                clone on $X$.
                \item $\sigma$ maps the clone of all operations on
                $A$ to $\ppol(\{A\})$.
                \item All local clones (in fact: all clones) which
                contain $\sigma(\{f: f \text{ is a}\\\text{projection on } A\})$ (i.e., which contain the local clone on $X$ which, via $\sigma$, corresponds to the
                clone of projections on $A$) and
                which are contained in $\ppol(\{A\})$ are of the
                form $\sigma(\C)$ for some clone $\C$ on $A$.
                \item $\sigma$ is one-to-one, and both $\sigma$ and its inverse are order preserving.
            \end{enumerate}
            (1) and (2) are easy verifications and left to the
            reader. To see (3), let $\D$ be any clone in the
            mentioned interval, and denote by $\C$ the set of all
            restrictions of operations in $\D$ to appropriate powers of
            $A$. Since $\D\subseteq\ppol(\{A\})$, all such
            restrictions are operations on $A$, and since $\D$ is
            closed under composition and contains all projections, so does
            $\C$. Thus, $\C$ is a clone on $A$. We claim
            $\D=\sigma(\C)$. By the definitions of $\C$ and $\sigma$, we have that $\sigma(\C)$ clearly contains
            $\D$. To see the less obvious inclusion, let
            $f\in\sigma(\C)$ be arbitrary, say of arity
            $m$. The restriction of $f$ to $A^m$ is an element of
            $\C$, hence there exists an $m$-ary $f'\in\D$ which has the same
            restriction to $A^m$ as $f$. Define
            $s(x_1,\ldots,x_m,y)\in\O^{(m+1)}$ by
            $$
                s(x_1,\ldots,x_m,y)=\begin{cases} y &,\text{ if}
                (x_1,\ldots,x_m)\in A^m\\ f(x_1,\ldots,x_m)& ,\text{
                otherwise.}\end{cases}
            $$
            Since $s$ behaves on $A^{m+1}$ like the projection
            onto the last coordinate, and since $\D$ contains $\sigma(\{f: f \text{ is } \text{ a } \text{ projection } \text{ on } A\})$, we infer $s\in\D$. But
            $f(x_1,\ldots,x_m)=s(x_1,\ldots,x_m,f'(x_1,\ldots,x_m))$,
            proving $f\in\D$.\\
            (4) is an immediate consequence of the
            definitions.
        \end{proof}

        It is known that all countable products of finite lattices
        embed into the clone lattice over a four-element set
        \cite{Bul94}, so by the preceding proposition they also
        embed into $\Cll(X)$. However, there are quite simple
        countable lattices which do not embed into the clone lattice
        over any finite set: The lattice $M_\omega$ consisting of a
        countably infinite antichain plus a smallest and a greatest element
        is an example \cite{Bul93}. We shall see now that the class
        of lattices embeddable into $\Cll(X)$ properly contains the
        class of lattices embeddable into the clone lattice over a
        finite set. In fact, the structure of $\Cll(X)$ is at least as
        complicated as the structure of any algebraic lattice with
        $\aleph_0$ compact elements. Before we prove this, observe that similarly to the local clones, the set of locally closed (that is: topologically closed in the space of all permutations on $X$) permutation groups on $X$ forms a complete lattice with respect to inclusion; denote this lattice by $\Grl(X)$. Moreover, the set of locally closed (that is: closed in the space $X^X$, where $X$ is discrete) transformation monoids on $X$ forms a complete lattice with respect to inclusion, which we denote by $\Monl(X)$. Note that the elements of $\Monl(X)$ are precisely the objects of the form $\C\cap\Oo$, where $\C\in\Cll(X)$.

        \begin{thm}\label{thm:algebraicIntervals}
            Every algebraic lattice with a countable number of
            compact elements is a complete sublattice of $\Cll(X)$. In fact, every such lattice is a complete sublattice of $\Grl(X)$ and of $\Monl(X)$.
        \end{thm}

	    It is clear that the lattice $\Monl(X)$ of locally closed transformation monoids embeds completely into $\Cll(X)$ via the assignment which sends every local monoid to the local clone it generates. Local clones arising in this way will contain only operations which depend on at most one variable: In fact, the operations of such a local clone are exactly the functions of the monoid, with (possibly) some fictitious variables added. Therefore, the statement of the theorem about the lattice $\Monl(X)$ implies the statement about $\Cll(X)$.

        To prove Theorem \ref{thm:algebraicIntervals}, we cite
        the following deep theorem from
        \cite{Tum89subgroupLattices}.

        \begin{thm}\label{thm:tuma}
            Every algebraic lattice with a countable number of
            compact elements is isomorphic to an interval in the
            subgroup lattice of a countable group.
        \end{thm}

        \begin{proof}[Proof of Theorem \ref{thm:algebraicIntervals}]
			We prove the statement about $\Grl(X)$.
            Let $\L$ be the algebraic lattice to be embedded into
            $\Grl(X)$. Let $\X=(X,+,-,0)$ be the group provided by
            Theorem~\ref{thm:tuma}. Let $[\G_1,\G_2]$ be the interval in the
            subgroup lattice of $\X$ that
            $\L$ is isomorphic to. We will assign to every group in the interval its Cayley representation as a group of permutations on $X$: That is,
            for every $a\in X$, define a unary operation $f_a\in\Oo$ by
            $f_a(x)=a+x$. Clearly, we have
            $f_a(f_b(x))=a+b+x=f_{a+b}(x)$ for all $a,b\in X$.
             Define a mapping $\mu:[\G_1,\G_2]\rightarrow \Grl(X)$ sending
            every group
            $\H=(H,+,-,0)$ in the interval to $\C_H:=\{f_a:a\in H\}$.  It is easy to see (and folklore) that the $\C_H$ are permutation groups; we only have to check that they are locally closed. To see this, let
            $f\in\Oo$ be an element of the topological (local) closure of ${\C_H}$ in the full symmetric group on $X$. We claim $f\in \C_H$.
			Indeed, observe
            that $f$ agrees with some $f_a\in \C_H$
            on the finite set
            $\{0\}\subseteq X$. Suppose that there is $b\in X$ such
            that $f(b)\neq f_a(b)=a+b$. Then take any $f_c\in \C_H$ such that $f$ and $f_c$ agree on $\{0,b\}$.
			But then
            $c=f_c(0)=f(0)=f_a(0)=a$, and thus $f(b)=f_c(b)=f_a(b)\neq f(b)$, an obvious contradiction. Hence,
            $f=f_a\in \C_H$ and we are done.

            With the explicit description of the $\C_H$ and given that they are indeed closed permutation groups, a straightforward check
            shows that $\mu$ preserves arbitrary meets and joins.
            
            The proof for the embedding into $\Monl(X)$ is identical. By the discussion above the statement about $\Cll(X)$ follows.
        \end{proof}

        Since in particular, $\Cll(X)$ contains $M_\omega$ as a
        sublattice, and since according to \cite{Bul93}, $M_\omega$ is not a sublattice of the
        clone lattice over any finite set, we have the following
        corollary to Theorem \ref{thm:algebraicIntervals}.

        \begin{cor}\label{cor:localnotinfinite}
            $\Cll(X)$ does not embed into the clone lattice over any
            finite set.
        \end{cor}

        Observe also that Theorem \ref{thm:algebraicIntervals} is a
        strengthening of Proposition \ref{prop:finiteintolocal} in
        so far as the clone lattice over a finite set is an example
        of an algebraic lattice with countably many compact
        elements. However, in that proposition we obtain an
        embedding as an interval, not just as a complete sublattice.

        What about other lattices, i.e., lattices which are more
        complicated or larger than algebraic lattices with countably
        many compact elements? We will now find a restriction
        on which lattices can be sublattices of $\Cll(X)$.

        A \emph{partial clone of finite operations} on $X$ is a set of partial
        operations of finite domain
        on $X$ which contains all restrictions of the projections to finite domains and which is closed
        under composition. A straightforward verification shows that the set of partial clones of finite operations on $X$
        forms a complete algebraic lattice $\Clp(X)$, the compact elements of
        which are precisely the finitely generated partial clones; in particular, the number of compact elements of $\Clp(X)$ is countable.

        \begin{thm}\label{thm:partialClones}
            The mapping $\sigma$ from $\Cll(X)$ into $\Clp(X)$ which sends every $\C\in\Cll(X)$
            to the partial clone of all restrictions of its operations
            to finite
            domains is one-to-one and preserves arbitrary joins.
        \end{thm}
        \begin{proof}
            It is obvious that $\sigma(\C)$ is a partial clone of
            finite operations, for all local (in fact: also non-local) clones $\C$.\\
            Let $\C,\D\in\Cll(X)$ be distinct. Say without loss of generality that there is an
            $n$-ary
            $f\in\C\setminus\D$; then since $\D$ is locally closed, there exists some finite set
            $A\subseteq X^n$ such that there is no $g\in \D$ which
            agrees with $f$ on $A$. The restriction of $f$ to $A$
            then witnesses that $\sigma(\C)\neq\sigma(\D)$.\\
            We show that
            $\sigma(\C)\vee\sigma(\D)=\sigma(\C\vee\D)$; the proof for arbitrary
            joins works the same way. It follows directly from the
            definition of $\sigma$ that it is
            order-preserving. Thus, $\sigma(\C\vee\D)$ contains
            both $\sigma(\C)$ and $\sigma(\D)$ and hence also their
            join. Now let $f\in \sigma(\C)\vee\sigma(\D)$. This
            means that it is a composition of partial operations in
            $\sigma(\C)\cup\sigma(\D)$. All partial operations used
            in this composition have extensions to operations in
            $\C$ or $\D$, and if we compose these extensions in the
            same way as the partial operations, we obtain an
            operation in $\C\vee\D$ which agrees with $f$ on the
            domain of the latter. Whence, $f\in\sigma(\C\vee\D)$.
        \end{proof}

        \begin{cor}\label{cor:intopowersetofomega}
            $\Cll(X)$ embeds as a suborder into the power set of $\omega$. In
            particular,  the size of $\Cll(X)$ is $2^{\aleph_0}$, and $\Cll(X)$ does not contain any uncountable
            ascending or descending chains.
        \end{cor}
        \begin{proof}
			The number of partial operations with finite
        domain on $X$ is countable; therefore, partial clones of finite operations can be considered as
        subsets of $\omega$, which proves the first statement. In particular, $\Cll(X)$ cannot have more than $2^{\aleph_0}$ elements, and the fact that all algebraic lattices with at most
            $\aleph_0$ compact elements embed into $\Cll(X)$ shows
            that it must contain at least $2^{\aleph_0}$ elements. The third statement is an immediate consequence of the first one.
        \end{proof}

 It also follows from Theorem~\ref{thm:partialClones} that for all sets of local clones $S\subseteq \Cll(X)$, there exists  a countable $S'\subseteq S$ such that $\bigvee S=\bigvee S'$.

		It might be interesting to observe at this point that $\Clp(X)$ is universal for the class of algebraic lattices with countably many compact elements: Denote by $\Monp(X)$ the lattice of all sets of finite partial unary operations on $X$ which are closed under composition, and which contain all restrictions of the identity. In other words, the elements of $\Monp(X)$ are precisely the unary fragments of the elements of $\Clp(X)$. As before with local monoids and local clones, $\Monp(X)$ embeds naturally into $\Clp(X)$ by adding fictitious variables to the partial operations of an element of $\Monp(X)$. We have:

		\begin{prop}
			Every algebraic lattice with a countable number of
            compact elements is a complete sublattice of $\Clp(X)$. In fact, every such lattice is a
			complete sublattice of $\Monp(X)$.
		\end{prop}
		\begin{proof}
			Let $\L$ be the lattice to be embedded, and embed it into $\Monl(X)$ as in the proof of Theorem~\ref{thm:algebraicIntervals} via the mapping $\mu$ of that proof. Then apply the mapping $\sigma$ from Theorem~\ref{thm:partialClones} to the local monoids thus obtained. This obviously gives us a join-preserving mapping from $\L$ into $\Monp(X)$. We claim that $\sigma$, restricted to monoids in the image of $\mu$, preserves arbitrary meets: Let $\C_H, \C_K$ be two such monoids. Clearly, $\sigma(\C_H\cap \C_K)$ is contained in $\sigma(\C_H)\cap \sigma(\C_K)$. Now let $p\in \sigma(\C_H)\cap \sigma(\C_K)$. Then, using the notation from the proof of Theorem~\ref{thm:algebraicIntervals}, there exist a finite set $A\subseteq X$, $a\in H$ and $b\in K$ such that $p=f_a\rest_A= f_b\rest_A$. We may assume that $A$ is non-void; pick $c\in A$. We have $p(c)=f_a(c)=a+c=f_b(c)=b+c$. Hence, $a=b$, so $f_a=f_b\in \C_H\cap\C_K$, and $p\in\sigma(\C_H\cap \C_K)$. Larger meets work the same way, and thus we have a complete lattice embedding of $\L$ into $\Monp(X)$.
		\end{proof}

        Until today, no other restriction to embeddings into
        $\Cll(X)$ except for Theorem~\ref{thm:partialClones} is known, and we ask:

        \begin{prob}\label{prob:powerOmega}
            Does every lattice which has a complete join-embedding into an algebraic lattice with
			countably many compacts have a lattice embedding into $\Cll(X)$?
        \end{prob}

        It seems, however, difficult to embed even the simplest
        lattices which are not covered by Theorem
        \ref{thm:algebraicIntervals} into $\Cll(X)$. In the conference version of this paper, \cite{Pin08sublatticesOfTheLocal}, the following problem was posed:

        \begin{quote}
            Does the lattice $M_{2^\omega}$, which consists of an antichain of length $2^{\omega}$ plus a smallest and a largest element, embed
            into $\Cll(X)$?
        \end{quote}

	We will give an affirmative answer to this problem in Section~3. This is the first example of a lattice which embeds into $\Cll(X)$ but not into any algebraic lattice with countably many compacts (for the latter statement, see the proof of Corollary~\ref{cor:notIntoAlg}).

It is much easier to see the following:

        \begin{prop}
            There exist a join-preserving embedding as well as a meet-preserving embedding of
            $M_{2^\omega}$ into $\Cll(X)$.
        \end{prop}
        \begin{proof}
            Denote by $0$ and $1$ the smallest and the largest element of $M_{2^\omega}$, respectively, and enumerate the elements of its antichain by
            $(a_i)_{i\in {2^{\omega}}}$.\\ We first construct a join-preserving embedding.
            Enumerate the non-empty proper subsets of $X$ by $(A_i)_{i\in 2^{\omega}}$. Consider the mapping $\sigma$ which sends
            $0$ to the clone of projections, $1$ to $\O$, and every
            $a_i$ to $\ppol(\{A_i\})$. Now it is well-known (see \cite{RS84-LocalCompletenessI}) that for any non-empty proper subset $A$ of $X$,
            $\ppol(\{A\})$ is
            covered by $\O$, i.e.\ there exist no local (in fact even no global) clones between $\ppol(\{A\})$ and $\O$. Hence, we have that
            $\sigma(a_i)\vee\sigma(a_j)=\cll{\ppol(\{A_i\})\cup\ppol(\{A_j\})}=\O=\sigma(1)$ for all $i\neq j$. Since clearly $\sigma(a_i)$ contains $\sigma(0)$ for all $i\in 2^{\omega}$, the mapping $\sigma$
            indeed preserves joins.\\
            To construct a meet-preserving embedding, fix any distinct $a,b\in X$ and
            define for every non-empty subset $A$ of $X\setminus\{a,b\}$ an
            operation $f_A\in\Oo$ by
            $$
                f_A(x)=\begin{cases}a, &\text{if } x\in A\\b,
                &\text{otherwise.}
                \end{cases}
            $$
            Enumerate the non-empty subsets of $X\setminus\{a,b\}$ by $(B_i:i\in 2^{\omega})$.
            Denote the constant unary operation with value $b$ by $c_b$. Let the embedding $\sigma$ map $0$ to
            $\cll{\{c_b\}}$, for all $i\in 2^{\omega}$ map $a_i$ to $\cll{\{f_{B_i}\}}$, and
            let it map $1$ to $\O$. One readily checks that
            $\sigma(a_i)=\cll{\{f_{B_i}\}}$ contains only
            projections and, up to fictitious variables, the operations $f_{B_i}$ and $c_b$.
            Therefore, for $i\neq j$ we have $\sigma(a_i)\wedge
            \sigma(a_j)=\cll{\{c_b\}}=\sigma(0)$. Since clearly
            $\sigma(a_i)\subseteq \sigma(1)=\O$ for all $i\in
            2^{\omega}$, we conclude that $\sigma$ does indeed
            preserve meets.

        \end{proof}

        Simple as the preceding proposition is, it still shows us as a
        consequence that Theorem \ref{thm:algebraicIntervals} is not
        optimal.

        \begin{cor}\label{cor:notIntoAlg}
            $\Cll(X)$ is not embeddable into any algebraic lattice
            with countably many compact elements.
        \end{cor}
        \begin{proof}
            It is well-known and easy to check (confer also \cite{CD73}) that any algebraic lattice $\L$ with countably many compact elements can be
            represented as the subalgebra lattice of an algebra over
            the base set
            $\omega$. The meet in the subalgebra lattice $\L$ is just
            the set-theoretical intersection. Now there is certainly
            no uncountable family of subsets of $\omega$ with the
            property that any two distinct members of this family have
            the same intersection $D$; for the union of such a family would have to be uncountable. Consequently, $\L$ cannot
            have $M_{2^\omega}$ as a meet-subsemilattice. But $\Cll(X)$ has, hence $\L$ cannot have $\Cll(X)$ as a sublattice.
        \end{proof}

        Observe that this corollary is a strengthening of Corollary
        \ref{cor:localnotinfinite}, since the clone lattice over a
        finite set is an algebraic lattice with countably many
        compact elements.

        We conclude this section by remarking that the lattice $\Cl(X)$ of all (not
        necessarily local) clones on $X$ is infinitely more
        complicated than $\Cll(X)$: It contains all algebraic
        lattices with at most $2^{\aleph_0}$ compact elements, and
        in particular all lattices of size continuum (including $\Cll(X)$), as complete
        sublattices \cite{Pin06AlgebraicSublattices}.

\section{How to embed $M_{2^{\omega}}$ into $\Cll(X)$}

	This section is devoted to the proof of the following theorem.

	\begin{thm}\label{thm:mcont}
		$M_{2^{\omega}}$ is isomorphic to a sublattice of $\Cll(X)$.
	\end{thm}

	Write
    $X=A\dot{\cup}B\dot{\cup}\{\infty\}$, with both $A$ and $B$ infinite.
    Without loss of generality assume $B=2^{<\omega}$, the set of all finite $0$-$1$-sequences. For every infinite $0$-$1$-sequence $\alpha\in 2^\omega$, write
    $B_\alpha:=\{\alpha\rest_n:n<\omega\}$. Clearly, the $B_\alpha$ form an almost
    disjoint family of infinite subsets of $B$, that is, $B_\alpha\cap B_\beta$ is finite whenever $\alpha,\beta\in 2^\omega$ are distinct.
    For all $\alpha\in 2^\omega$, set
    $\GG_\alpha\subseteq\Oo$ to consist of all functions $f\in\Oo$
    satisfying the following two properties:
    \begin{itemize}
        \item $f$ maps $\{\infty\}\cup B$ to $\infty$, and
        \item $f$ maps $A$ injectively into $B_\alpha$.
    \end{itemize}

     For $c,d\in B$, we write $c\perp d$ iff none of the two elements is an initial segment of the
     other, i.e.\ iff there is no $\alpha$ such that $c,d\in B_\alpha$. If $C,D\subseteq B$, we write $C\perp D$ iff $c\perp d$
     for all $(c,d)\in C\mult D$. For a function $\sigma$, write $\dom(\sigma)$ and $\ran(\sigma)$ for the domain and the range of $\sigma$, respectively. Now for all functions $\sigma$
     with the properties
     \begin{itemize}
        \item $\dom(\sigma)=C\mult D$ for some finite $C,D\subseteq B$ with $C\perp D$,
        and
        \item $\ran(\sigma)\subseteq B$, and
        \item $\sigma$ is injective,
     \end{itemize}
     define an operation $m^\sigma\in \O^{(2)}$ by
    $$
        m^\sigma(x,y):=
            \begin{cases}
                                    \infty&,\, x=\infty\vee y=\infty\\
                                    y&,\, x\in A\wedge y\in B\\
                                    \sigma(x,y)&,\, (x,y)\in
                                    \dom(\sigma)\\
                                    x&,\text{otherwise.}\\
            \end{cases}
    $$

    In words, $m^\sigma$ does the following: If one of its
    arguments equals $\infty$, then it returns $\infty$. If one of the
    arguments is in $A$ and the other one in $B$, then it returns the one in
    $B$. If both arguments are in $B$ and $\sigma$ is defined for $(x,y)$, then $m^\sigma(x,y)=\sigma(x,y)$. Otherwise, $m^\sigma$
    returns the first argument.

    Denote by $\M$ the set of all such operations
    $m^\sigma$, together with the constant function with value $\infty$, which we denote by $\infty$ as well.
    For all $\alpha\in 2^\omega$, set $\D_\alpha:=\cl{\M\cup\GG_\alpha}$, which is to denote the normal clone closure of $\M\cup\GG_\alpha$, i.e., the set of all operations which can be written as a term of the operations in $\M\cup\GG_\alpha$ and the projections. Set moreover $\C_\alpha:=\cll{\M\cup\GG_\alpha}$ (the topological closure of $\D_\alpha$), for all
    $\alpha\in 2^\omega$. Finally, set $\C:=\cll{\bigcup_{\alpha\in 2^\omega}\C_\alpha}$.

We now aim at proving $\cll{\C_\alpha\cup\C_\beta}=\C$ and $\C_\alpha\cap\C_\beta=\cll{\M}$, for all distinct $\alpha,\beta\in2^\omega$, which clearly implies our theorem.

    \begin{lem}\label{lem:composition}
        Let $m\in\M$, and $f,g \in\GG_\alpha$ for some $\alpha\in 2^\omega$. Then the
        following equations hold:
            $$m(x,x)=x\quad \text{and}\quad m(f(x),x)=m(x,f(x))=m(f(x),g(x))=f(x).$$
    \end{lem}
    \begin{proof}
        We leave the straightforward verification to the reader.
    \end{proof}
	
Denote the identity operation on $X$ by $\id$.

    \begin{lem}
        Let $f\in\C_\alpha\uo$. Then $f\in\GG_\alpha$, or $f=\infty$, or $f=\id$.
    \end{lem}
    \begin{proof}
        We first prove the statement for all $f\in\D_\alpha\uo$, using
        induction over terms. The beginning is trivial, so let
        $f=g(t)$, with $g\in \GG_\alpha\cup\{\id\}$ and $t\in\D_\alpha\uo$ satisfying
        the induction hypothesis. There is nothing to show if $t$ is the identity or $\infty$.
        If $t\in\GG_\alpha$, then also $f\in\GG_\alpha$ if $g$ is the identity; otherwise, $g\in\GG_\alpha$ and $f=\infty$.
        Now assume $f=m(s,t)$, with $m\in\M$ and $s,t$
        satisfying the induction hypothesis. Then the preceding lemma immediately implies our assertion.

        Now with this description of $\D_\alpha\uo$ it is easy to check that $\D_\alpha\uo$ is locally closed. Hence $\C_\alpha\uo=\D_\alpha\uo$,
        proving the lemma.
    \end{proof}

    \begin{lem}\label{lem:identifyingVariables}
        Let $t(x_1,\ldots,x_n)\in\D_\alpha$, and assume that $t$ has a
        representation as a term over $\GG_\alpha\cup\M$ (without projections!) which uses at
        least one symbol from $\GG_\alpha$. Then either $t=\infty$, or $t(x,\ldots,x)\in\GG_\alpha$.
    \end{lem}
    \begin{proof}
        We use induction over the
        complexity of $t$. The beginning is trivial, so write
        $t=g(s)$, where $g\in\GG_\alpha$ and $s$ satisfies the induction
        hypothesis. Assume first that $s$ does not contain any symbol
        from $\GG_\alpha$; then $s(x,\ldots,x)=x$ by Lemma
        \ref{lem:composition} and a standard induction, implying
        $t(x,\ldots,x)=g(s(x,\ldots,x))=g(x)\in\GG_\alpha$. If on the
        other hand $s$ does contain a symbol from $\GG_\alpha$, then
        the range of $s$ is contained in $B\cup\{\infty\}$, as is
        easily verified by a straightforward induction using Lemma~\ref{lem:composition}. Hence,
        $g(s)=\infty$. Now write
        $t=m(r,s)$, with $m\in\M$ and $r,s$ satisfying the
        induction hypothesis. Then our assertion follows from Lemma \ref{lem:composition} and the definition of the operations in $\M$.
    \end{proof}

    \begin{lem}\label{lem:meet}
        Let $\alpha,\beta\in 2^\omega$ be distinct. Then $\C_\alpha\cap\C_\beta=\cll{\M}$.
    \end{lem}
    \begin{proof}
        It suffices to show $\C_\alpha\cap\C_\beta\subseteq\cll{\M}$. To see
        this, let $t(x_1,\ldots,x_n)\in \C_\alpha\cap\C_\beta$, and
        suppose there is a finite set $F\subseteq X$ such that no
        operation in $\cl{\M}$ (normal clone closure) agrees with $t$ on $F^n$. By
        expanding $F$, we may assume that $|F\cap A|>|B_\alpha\cap
        B_\beta|$ (since the latter set is finite).
        Let $s$ be
        a term over $\GG_\alpha\cup\M$ which interpolates $t$ on
        $F^n$. By our assumption, there appears some function from
        $\GG_\alpha$ in $s$. Then by Lemma
        \ref{lem:identifyingVariables}, $s(x,\ldots,x)\in\GG_\alpha$.
        Therefore, also $t(x,\ldots,x)$ behaves like an operation
        from $\GG_\alpha$ on $F$, and in particular on $F\cap A$.
        Hence, it maps $F\cap A$ injectively into $B_\alpha$. Now
        the same argument shows us that $t(x,\ldots,x)$ maps $F\cap
        A$ injectively into $B_\beta$, hence it maps $F\cap A$
        injectively into $B_\alpha\cap B_\beta$, in contradiction with the size of these sets.
    \end{proof}

    \begin{lem}\label{lem:join}
        Let $\alpha,\beta\in 2^\omega$ be distinct. Then $\cll{\C_\alpha\cup\C_\beta}=\C$.
    \end{lem}
    \begin{proof}
        Let $\gamma\in 2^\omega$, $\gamma \nin \{\alpha,\beta\}$ be arbitrary, and consider any $h\in\GG_\gamma$. It
        suffices to show that $h\in\cll{\C_\alpha\cup\C_\beta}$. So let
        $F\subseteq X$ be finite. We have to find an operation in
        $\cl{\C_\alpha\cup\C_\beta}$ which agrees with $h$ on $F$. Assume
        for the moment that $F\subseteq A$. Pick any $C\subseteq
        B_\alpha$, $D\subseteq B_\beta$ with $|C|=|D|=|F|$ and such
        that $C\perp D$. Pick $f\in \GG_\alpha$ mapping $F$ onto $C$,
        and $g\in \GG_\beta$ mapping $F$ onto $D$. Then $(f,g): X\To
        X^2$ maps $F$ injectively into $C\mult D$. Thus there exists
        $\sigma: C\mult D\To B_\gamma$ such that $\sigma(f(x),g(x))=h(x)$
        for all $x\in F$. Now
        $m^\sigma(f(x),g(x))=\sigma(f(x),g(x))=h(x)$ for all $x\in F$,
        so $m^\sigma(f(x),g(x))\in \cl{\C_\alpha\cup\C_\beta}$ agrees with $h$ on $F$.\\
        For the case where $F$ is not a subset of $A$, one constructs the interpolation for $F':=F\cap A$. Then observe that
        the operation $m^\sigma(f(x),g(x))$ constructed sends all $b\in B\cup\{\infty\}$ to
        $\infty$, thus behaving like $h$ outside $A$ anyway.
    \end{proof}

Lemmas~\ref{lem:meet} and \ref{lem:join} clearly prove Theorem~\ref{thm:mcont}.

\section{Open problems}

Suprisingly, the method of the previous section seems to be hard to generalize: For example, we do not know:

\begin{prob}\label{prob:finiteHeight}
	Does every lattice of finite height which has cardinality $2^{\aleph_0}$ embed into $\Cll(X)$?
\end{prob}

When proving that all algebraic lattices with countably many compacts embed into $\Cll(X)$, we in fact embedded them into $\Monl(X)$ (Theorem~\ref{thm:algebraicIntervals}). The same could be true with $M_{2^\omega}$ (which we do not believe, though):

\begin{prob}\label{prob:Mcont}
	Does $M_{2^\omega}$ embed into $\Monl(X)$? Does it embed into $\Grl(X)$?
\end{prob}

In theory, even the following could hold:

\begin{prob}\label{prob:monoid}
	Does $\Cll(X)$ embed into $\Monl(X)$? Does it embed into $\Grl(X)$?
\end{prob}

%\bibliographystyle{alpha}
%\bibliography{order_bib}

\def\ocirc#1{\ifmmode\setbox0=\hbox{$#1$}\dimen0=\ht0 \advance\dimen0
  by1pt\rlap{\hbox to\wd0{\hss\raise\dimen0
  \hbox{\hskip.2em$\scriptscriptstyle\circ$}\hss}}#1\else {\accent"17 #1}\fi}

\end{document}